\documentclass[10pt]{article}
\usepackage{amssymb,amsmath,amsthm,youngtab,verbatim}
\newtheorem*{definition}{Definition}

\newtheorem{claim}{Claim}
\newtheorem{theorem}{Theorem}
\newtheorem{lemma}[theorem]{Lemma}
\newtheorem{fact}{Fact}
\newtheorem{conjecture}[theorem]{Conjecture}
\newtheorem{proposition}[theorem]{Proposition}
\newtheorem{corollary}[theorem]{Corollary}
\DeclareMathOperator{\sgn}{sign}
\DeclareMathOperator{\ra}{\rangle}
\DeclareMathOperator{\la}{\langle}

\DeclareMathOperator{\N}{\mathbb{N}}
\begin{document}
\title{An isoperimetric inequality for conjugation-invariant sets in the symmetric group}
\author{Neta Atzmon\footnote{Faculty of Mathematics and Computer Science, Weizmann Institute of Science, Israel.},\, David Ellis\footnote{School of Mathematical Sciences, Queen Mary, University of London, UK. Research supported in part by a Feinberg Visiting Fellowship from the Weizmann Institute of Science.}\, and Dmitry Kogan\footnote{Faculty of Mathematics and Computer Science, Weizmann Institute of Science, Israel.}}
\maketitle
\begin{abstract}
We prove an isoperimetric inequality for conjugation-invariant sets of size $k$ in $S_n$, showing that these necessarily have edge-boundary considerably larger than some other sets of size $k$ (provided $k$ is small). Specifically, let $T_n$ denote the Cayley graph on $S_n$ generated by the set of all transpositions. We show that if $A \subset S_n$ is a conjugation-invariant set with $|A| = pn! \leq n!/2$, then the edge-boundary of $A$ in $T_n$ has size at least
$$c \cdot \frac {\log_2 \left( \tfrac 1{p} \right)
    }{\log_2 \log_2 \left( \tfrac 2{p} \right)}\cdot n \cdot |A|,$$
    where $c$ is an absolute constant. (This is sharp up to an absolute constant factor, when $p = \Theta(1/s!)$ for any $s \in \{1,2,\ldots,n\}$.) It follows that if $p = n^{-\Theta(1)}$, then the edge-boundary of a conjugation-invariant set of measure $p$ is necessarily a factor of $\Omega(\log n / \log \log n)$ larger than the minimum edge-boundary over all sets of measure $p$.
\end{abstract}

\section{Introduction}
Isoperimetric problems are classical objects of study in mathematics. In general, they ask for the smallest possible `boundary' of a set of a given `size'. For example, of all shapes in the plane with area 1, which has the smallest perimeter? The ancient Greeks were sure that the answer is a circle, but it was not until the 19th century (with the work of Weierstrass) that this was proved rigorously.

In the last fifty years, `discrete' isoperimetric problems have been extensively studied. These deal with the boundaries of sets of vertices in graphs. Here, there are two different notions of boundary. If \(G=(V,E)\) is a graph, and \(A \subset V\), the {\em vertex-boundary of \(A\) in \(G\)} is the set of all vertices in \(V \setminus A\) which have a neighbour in \(A\). (This is sometimes denoted by $b_{G}(A)$.) Similarly, the {\em edge-boundary of \(A\) in \(G\)} is the set of all edges of \(G\) between \(A\) and \(V \setminus A\). (This is often denoted by $\partial_G(A)$.) The {\em vertex-isoperimetric problem for \(G\)} asks for the minimum possible size of the vertex-boundary of a \(k\)-element subset of \(V\), for each \(k \in \mathbb{N}\). Similarly, the {\em edge-isoperimetric problem for \(G\)} asks for the minimum possible size of the edge-boundary of a \(k\)-element subset of \(V\), for each \(k \in \mathbb{N}\).

A well-known example arises from taking the graph \(G\) to be the \(n\)-dimensional hypercube \(Q_n\), the graph with vertex-set \(\{0,1\}^n\), where \(x\) and \(y\) are joined by an edge if and only if they differ in exactly one coordinate. It turns out that for any $k \in \{1,2,\ldots,2^{n}\}$, the edge-boundary of a \(k\)-element subset of $\{0,1\}^n$ is minimized by taking the first \(k\) elements of the binary ordering on \(\{0,1\}^n\). (This was proved by Harper \cite{harper}, Lindsey \cite{lindsey}, Bernstein \cite{bernstein}, and Hart \cite{hart}.) It follows that if $A \subset \{0,1\}^n$ with $|A|=2^{n-t}$, where $t \in \{1,2,\ldots,n\}$ then $|\partial_{Q_n}(A)| \geq t2^{n-t}$. Equality holds if and only if $A$ is a subcube of codimension $t$.

The reader is referred to \cite{leader} for a survey of results and open problems in the field of discrete isoperimetric inequalities.

It is natural to ask what happens to the minimum size of the edge-boundary if one imposes some kind of symmetry constraint on the set $A$. For example, we say that a set $A \subset \{0,1\}^n$ is {\em transitive-symmetric} if there exists a transitive subgroup $H \leq S_n$ such that $\sigma(A) = A$ for all $\sigma \in H$. (Here, $\sigma(A) := \{\sigma(x):x \in A\}$, where $\sigma(x)$ is defined by $(\sigma(x))_i = x_{\sigma^{-1}(i)}$ for each $i \in [n]$.) In other words, $A$ is transitive-symmetric if there exists a group which acts transitively on the coordinates and leaves $A$ invariant. It turns out that a transitive-symmetric set in $\{0,1\}^n$ must have considerably larger edge-boundary than some other sets of the same size. This follows from the celebrated KKL theorem on influences. Recall that if $A \subset \{0,1\}^n$, and $i \in [n]$, the {\em influence} $\textrm{Inf}_i(A)$ of the $i$th coordinate on $A$ is defined by
$$\textrm{Inf}_i(A) = \frac{|\{x \in \{0,1\}^n:\ \textrm{exactly one of }x\textrm{ and }x^i \textrm{ is in }A\}|}{2^{n}},$$
where $x^i$ denotes $x$ with the $i$th coordinate flipped. Equivalently, if $E_i(Q_n)$ denotes the set of all $2^{n-1}$ direction-$i$ edges of $Q_n$ (meaning, edges of the form $\{x,x^i\}$), then
$$\textrm{Inf}_i(A) = \frac{|\partial_{Q_n}(A) \cap E_i(Q_n)|}{|E_i(Q_n)|} = \frac{|\partial_{Q_n}(A) \cap E_i(Q_n)|}{2^{n-1}}.$$
Note that
$$|\partial_{Q_n}(A)| = 2^{n-1}\sum_{i=1}^{n} \textrm{Inf}_i(A).$$

Kahn, Kalai and Linial \cite{kkl} proved the following.
\begin{theorem}[Kahn, Kalai, Linial]
\label{thm:kkl}
Let $A \subset \{0,1\}^n$ with $|A| = p2^n$. Then provided $n$ is larger than an absolute constant, there exists a coordinate $i \in [n]$ such that
$$\textrm{Inf}_i(A) \geq p(1-p) \frac{\ln n}{n}.$$
\end{theorem}
If $A$ is transitive-symmetric, then all its influences are the same, so by Theorem \ref{thm:kkl}, its edge-boundary must satisfy
\begin{equation}\label{eq:trans-sym} |\partial_{Q_n} (A)| \geq 2^{n-1} p(1-p) \ln n,\end{equation}
provided $n$ is larger than an absolute constant. When $|A| = 2^{n-t}$ (so $p = 2^{-t}$), and $t \in \mathbb{N}$ is bounded, this is a factor of approximately $\ln n$ larger than the `unrestricted' minimum edge-boundary of $t \cdot p 2^n$, attained by a subcube. For $n^{-\Omega(1)} \leq p \leq 1-n^{-\Omega(1)}$, the `tribes' construction of Ben-Or and Linial \cite{tribes} gives a transitive-symmetric family $A \subset \{0,1\}^n$ with $|A| = pn!$ and with 
$$|\partial_{Q_n} (A)| = \Theta(2^{n-1} p(1-p) \ln n),$$
showing that (\ref{eq:trans-sym}) is sharp up to an absolute constant factor.

We study an analogue of this phenomenon for the symmetric group $S_n$, the group of all permutations of $\{1,2,\ldots,n\}$. Let $T_n$ denote the transposition graph on $S_n$. This is the Cayley graph on $S_n$ generated by the set of all transpositions, i.e. the graph with vertex-set $S_n$, where two permutations $\sigma,\pi \in S_n$ are joined by an edge if and only if $\sigma \pi^{-1}$ is a transposition. In other words, writing a permutation $\sigma \in S_n$ in sequence notation $(\sigma(1),\sigma(2),\ldots,\sigma(n))$, two permutations are joined by an edge if and only if their sequences differ by swapping two elements. We are interested in the edge-boundary of sets in $T_n$. If $A \subset S_n$, we let $\partial A = \partial_{T_n}(A)$ denote the edge-boundary of $A$ in $T_n$. We define the {\em lexicographic order} on $S_n$ by $\sigma < \pi$ if and only if $\sigma(j) < \pi(j)$, where $j = \min\{i \in [n]: \sigma(i) \neq \pi(i)|\}$. Ben Efraim \cite{ben-efraim} made the following conjecture.
\begin{conjecture}[Ben Efraim]
\label{conj:ben-efraim}
For any \(\mathcal{A} \subset S_n\), \(|\partial \mathcal{A}| \geq |\partial \mathcal{C}|\), where \(\mathcal{C}\) denotes the initial segment of the lexicographic order on \(S_n\) of size \(|\mathcal{A}|\). 
\end{conjecture}
(Here, the {\em initial segment of size $k$} of the lexicographic order means the first $k$ smallest elements of $S_n$ in the lexicographic order.) 

To date, Conjecture \ref{conj:ben-efraim} is known only for sets of size $c(n-1)!$ where $c \in \mathbb{N}$ (see Corollary \ref{corr:diaconis}), and for sets of size $(n-t)!$, where $n$ is sufficiently large depending on $t$ (see \cite{eff3}).

Note that for any $t \in [n]$, the initial segment of the lexicographic order of size $(n-t)!$ is precisely the set of all permutations in $S_n$ fixing $[t]$ pointwise, which has edge-boundary of size $t(n-1)(n-t)!$. Similarly, it can be checked (see Appendix) that if $A \subset S_n$ is an initial segment of the lexicographic ordering on $S_n$ with $(n-t-1)! < |A| \leq (n-t)!$ for some $t \in \{0,1,2,\ldots,n-1\}$, then
\begin{equation} \label{eq:unrestricted-bound}|\partial A| \leq (t+3/2)(n-1)|A|.\end{equation}

In this paper, we study the edge-boundary of subsets of $S_n$ which are conjugation-invariant, i.e. unions of conjugacy-classes of $S_n$. We feel that this is a natural invariance requirement to impose upon subsets of $S_n$. We prove the following.

\begin{theorem}
 \label{thm:main}
  There exists an absolute constant $c>0$ such that the following holds. Let $A\subset S_n$ be a conjugation-invariant family of permutations with $0 < |A| \leq n!/2$, and let $p = |A|/n!$ denote the measure of $A$. Then
  \[
  |\partial A| \ge c \cdot \frac {\log_2 \left( \tfrac{1}{p} \right)
    }{\log_2 \log_2 \left( \tfrac{2}{p} \right)}\cdot n \cdot |A|.
  \]
\end{theorem}

Note that an analogous result follows immediately for sets of size greater than $n!/2$, by applying the above result to $A^c$, since $\partial(A^c) = \partial A$.

Comparing the bound in the above theorem with (\ref{eq:unrestricted-bound}), we see that if $p = n^{-\Theta(1)}$, then the edge-boundary of a conjugation-invariant set of measure $p$ is necessarily a factor of $\Omega(\log n / \log \log n)$ larger than the minimum edge-boundary over all sets of measure $p$.

Observe that Theorem \ref{thm:main} is sharp up to the value of the absolute constant $c$, for a large number of different values of $p$. Indeed, for each $s \in [n]$, let
\begin{equation}\label{eq:fixed-points} A_s = \{\sigma \in S_n: \ \sigma \textrm{ has at least }s \textrm{ fixed points}\}.\end{equation}
We make the following.
\begin{claim}
For each $s \in [n-2]$,
\begin{equation}\label{eq:As-estimate} \frac{n!}{3s!} \leq |A_s| \leq \frac{n!}{s!}.\end{equation}
\end{claim}
\begin{proof}[Proof of claim.]
Recall that a {\em derangement} of $[m]$ is a permutation of $[m]$ with no fixed point. Let $d_{m}$ denotes the number of derangements of $[m]$. By the inclusion-exclusion formula, we have
$$d_m = \sum_{i=0}^{m} (-1)^i {m \choose i}(m-i)! = m!\sum_{i=0}^{m} (-1)^i \frac{1}{i!} \geq \frac{m!}{3} \quad \forall m \geq 2.$$
Note that ${n \choose s}d_{n-s}$ is precisely the number of permutations in $S_n$ with exactly $s$ fixed points. Hence, we have
$$\frac{n!}{3s!}= \tfrac{1}{3} {n \choose s}(n-s)! \leq {n \choose s}d_{n-s} \leq |A_s| \leq {n \choose s}(n-s)! = \frac{n!}{s!}\quad \forall s \in [n-2],$$
proving the claim.
\end{proof}

Hence, if $p = p_s  =|A_s|/n!$, then we have
$$\frac{1}{3s!} \leq p \leq \frac{1}{s!},$$
so
\begin{equation}
\label{eq:s-bound}
s = \Theta\left(\frac{\log_2 (\tfrac{1}{p})}{\log_2 \log_2 (\tfrac{2}{p})}\right).
\end{equation}
Note that (\ref{eq:s-bound}) also holds in the case $s \in \{n-1,n\}$, where $A_s = \{\textrm{Id}\}$.

Now observe that
$$|\partial (A_s)| \leq |A_s|(s+1)(n-1),$$
as an element $\sigma \in A_s$ is incident with at least one edge of $\partial (A_s)$ only if it has either $s$ or $s+1$ fixed points, and then there are at most $(s+1)(n-1)$ transpositions $\tau$ such that $\sigma \tau \notin A_s$. Putting everything together, we have
$$|\partial (A_s)| = \Theta\left(\frac{\log_2 (\tfrac{1}{p})}{\log_2 \log_2 (\tfrac{2}{p})}\right) \cdot n \cdot |A_s|,$$
confirming the sharpness of Theorem \ref{thm:main}.

Note that
$$|A_1| = n! - d_n = n! \left(1-\sum_{i=0}^{n}(-1)^{i}\frac{1}{i!}\right) = (1-1/e+o(1))n!,$$
and
$$|\partial A_1| \leq 2\cdot (n-1) \cdot |A_1|,$$
which is within an absolute constant factor of the lower bound
$$|\partial A| \geq (1/e)(1-1/e+o(1))\cdot n \cdot n!$$
given by plugging in $|A| = (1-1/e+o(1))n!$ into Corollary \ref{corr:diaconis} (see later). So for sets of constant measure, imposing the condition of conjugation-invariance cannot increase the minimum possible edge-boundary by more than a constant factor.

It is natural to ask what happens when one imposes a weaker condition than conjugation-invariance. We say that $A \subset S_n$ is {\em transitive-conjugation-invariant} if there exists a transitive subgroup $H \leq S_n$ such that $A$ is invariant under conjugation by any permutation in $H$ --- that is, for all $\sigma \in S_n$ and all $\pi \in H$, we have $\pi \sigma \pi^{-1} \in H$. However, it turns out that imposing this condition does not increase the minimum possible edge-boundary by more than an absolute constant factor, when $|A| = \Theta(\tfrac{n}{k}(n-k)!)$ for some $k \mid n$ (provided Conjecture \ref{conj:ben-efraim} holds). To see this, let $n,k \in \mathbb{N}$ with $k \mid n$. For each $i \in [n/k]$, let $I_i = \{(i-1)k+1,(i-1)k+2,\ldots,\ldots,ik\}$. Let
$$A = \{\sigma \in S_n:\ \sigma \textrm{ fixes some $I_i$ pointwise}\}.$$
Clearly, $A$ is transitive-conjugation-invariant; we may take the group $H$ to be the group of all permutations preserving the partition $I_1 \cup I_2 \cup \ldots \cup I_{n/k}$. In the case $k=1$, we have $A = A_1$ (as defined above), and in the case $k=n$, we have $A = A_n = \{\textrm{Id}\}$. Hence, we may assume that $1 < k \leq n/2$. We have
$$\frac{n}{2k} (n-k)! < \frac{n}{k}(n-k)! - {n/k \choose 2} (n-2k)! \leq |A| \leq \frac{n}{k} (n-k)!,$$
using the Bonferroni inequalities, so
$$(n-k)! < |A| \leq (n-k+1)!.$$
On the other hand, since a permutation fixing $I_i$ pointwise has at most $k(n-1)$ neighbours which do not fix $I_i$ pointwise, we have
$$|\partial A| \leq k(n-1)|A|.$$
This is within an absolute constant factor of the bound (\ref{eq:unrestricted-bound}) when $t=k-1$.

Our method of proving Theorem \ref{thm:main} is algebraic. We use the well-known expression
\begin{equation} \label{eq:quad-form} |\partial A| = 1_{A}^{\top} L 1_{A},\end{equation}
where $L$ denotes the Laplacian of $T_n$, and $1_{A}$ denotes the indicator function of the set $A \subset S_n$. (Of course, this holds when $T_n$ is replaced by any finite graph $G$, and $L$ is the Laplacian of $G$, for any subset $A \subset V(G)$.) We consider the expansion of the right-hand side of (\ref{eq:quad-form}) in terms of the eigenvalues of $L$ and the $L^2$-weights of $1_{A}$ on each eigenspace of $L$. We use known results to analyse the eigenvalues of $L$. Most of the work of our proof is in showing that if $A$ is conjugation-invariant, then most of the $L^2$-weight of $1_{A}$ is on eigenspaces of $L$ corresponding to `large' eigenvalues. (Here, the meaning of `large' depends on the size of the set $A$.) To do this, we use tools from the representation theory of the symmetric group.

\subsection*{Notation and background}
Before proving Theorem \ref{thm:main}, we first describe some notation and background.
\subsubsection*{Notation}
Throughout, we will write $\log(t)$ for $\log_2(t)$, and $\ln(t)$ for $\log_{e}(t)$. As usual, if $n \in \mathbb{N}$, we write $[n]$ for the set $\{1,2,\ldots,n\}$. If $X$ is a set, and $f,g : X \to \mathbb{R}$ are functions, we write $g = O(f)$ (and $f = \Omega(g)$) if there exists an absolute constant $C>0$ such that $|g(x)| \leq C|f(x)|$ for all $x \in X$. We write $g = \Theta(f)$ if $g = O(f)$ and $g = \Omega(f)$ both hold.
\subsubsection*{Background}
If $G = (V,E)$ is a finite graph, we define its {\em Laplacian} $L = L_{G}$ to be the matrix with rows and columns indexed by $V$, where
$$L_{u,v} = \begin{cases} d(v) & \mbox{ if }u=v;\\
-1 & \mbox{ if } u \neq v,\ \{u,v\} \in E(G);\\
0 & \mbox{ if } u \neq v,\ \{u,v\} \notin E(G).\end{cases}$$
For any $x \in \mathbb{R}^V$, we have
\begin{equation} \label{eq:quadratic-form} x^{\top}Lx = \sum_{\{u,v\} \in E(G)} (x(u) - x(v))^2,\end{equation}
so $L$ is a positive semidefinite matrix. Note that the constant vector $(1,1,\ldots,1) \in \mathbb{R}^V$ is always an eigenvector of $L$ with eigenvalue $0$. If $G$ is a connected graph, then $L$ has eigenvalue $0$ with multiplicity $1$. We write $\mu_2 = \mu_2(L)$ for the second-smallest eigenvalue of $L$. The following well-known theorem supplies a lower bound for the edge-boundary of a set $A \subset V(G)$, in terms of $\mu_2$.
\begin{theorem}[Alon, Milman \cite{alon-milman}]
\label{thm:alon-milman}
Let $G$ be a connected graph. If $A \subset V(G)$, then
$$|\partial_{G}(A)| \geq \mu_2 \frac{|A|(|V|-|A|)}{|V|}.$$
\end{theorem}
We give the standard proof (due to Alon and Milman), as we will need to refer to it later.
\begin{proof}
Let $1_{A} \in \{0,1\}^{V}$ denote the characteristic vector of $A$, defined by
$$1_{A}(v) = \begin{cases} 1 & \mbox{ if }v \in A;\\
0 & \mbox{ if } v \not\in A.\end{cases}$$
We equip $\mathbb{R}^V$ with the inner product
$$\langle x,y \rangle = \frac{1}{|V|} \sum_{v \in V} x(v) y(v).$$
Let $0 = \mu_1 < \mu_2 \leq \mu_3 \leq \ldots \leq \mu_{|V|}$ denote the eigenvalues of $L$, repeated with their multiplicities, and let $w_1 = (1,\ldots,1),w_2,\ldots,w_{|V|}$ be an orthonormal basis of $\mathbb{R}^V$ consisting of eigenvectors of $L$, such that $w_i$ is a $\mu_i$-eigenvector of $L$. Write
$$1_{A} = \sum_{i=1}^{|V|} b_i w_i$$
as a linear combination of the $w_i$. Then, by orthonormality, we have
$$|A|/|V| = \langle 1_{A},1_{A} \rangle = \sum_{i=1}^{|V|} b_i^2.$$
Moreover, we have
$$b_1 = \langle 1_{A},(1,\ldots,1)\rangle = |A|/|V|.$$
Using (\ref{eq:quadratic-form}), we have
\begin{align}
\label{eq:eval-expansion} |\partial_{G}(A)| & = 1_{A}^{\top} L 1_{A} \nonumber \\ 
& = |V|\langle 1_{A}, L 1_{A} \rangle \\
& = |V|\sum_{i=1}^{|V|} \mu_i b_i^2 \nonumber \\
& \geq |V| \mu_2 \sum_{i=2}^{|V|} b_i^2 \nonumber \\
& = |V| \mu_2 \left(\frac{|A|}{|V|} - \frac{|A|^2}{|V|^2}\right) \nonumber \\
& = \mu_2 \frac{|A|(|V|-|A|)}{|V|}, \nonumber \end{align}
proving Theorem \ref{thm:alon-milman}.
\end{proof}

We also need some background on the representation theory of $S_n$. This can be found, for example, in \cite{james-kerber}.

Let $L_n = L_{T_n}$ denote the Laplacian matrix of the transposition graph $T_n$. We equip $\mathbb{R}^{S_n}$ with the inner product
\begin{equation}\label{eq:inner-product} \langle x,y \rangle = \frac{1}{n!}\sum_{\sigma \in S_n} x(\sigma) y(\sigma),\end{equation}
and we let
$$\| x \| = \sqrt{\frac{1}{n!}\sum_{\sigma \in S_n} x(\sigma)^2}$$
denote the corresponding $L^2$-norm. Note that in the sequel, we will pass freely between vectors in $\mathbb{R}^{S_n}$ and the corresponding functions from $S_n$ to $\mathbb{R}$.

The eigenspaces of $L_n$ (and the corresponding eigenvalues) were determined by Diaconis and Shahshahani \cite{diaconis}. They are in a natural one-to-one correspondence with the irreducible characters\footnote{Recall that if $\Gamma$ is a finite group, an {\em irreducible character} of $\Gamma$ is a character of an irreducible representation of $\Gamma$.} of $S_n$ over $\mathbb{R}$, and in fact each irreducible character (when viewed as a vector) lies in the corresponding eigenspace. In turn, the irreducible characters of $S_n$ over $\mathbb{R}$ are in a natural one-to-one correspondence with the partitions of $n$. Recall the following.

\begin{definition}
If $n \in \mathbb{N}$, a {\em partition} of $n$ is a monotone non-increasing sequence of positive integers with sum $n$. In other words, $\alpha = (\alpha_1,\ldots,\alpha_l)$ is a partition of $n$ if $\alpha_i \in \mathbb{N}$ for all $i \in [l]$, $\alpha_1 \geq \alpha_2 \geq \ldots \geq \alpha_l$, and $\sum_{i=1}^{l}\alpha_i = n$. For example, $(3,2,2)$ is a partition of $7$. If $\alpha$ is a partition of $n$, then we sometimes write $\alpha \vdash n$. For each $n \in \mathbb{N}$, we write $p(n)$ for the number of partitions of $n$; for convenience, we define $p(0)=1$.
\end{definition}

If $\alpha = (\alpha_1,\ldots,\alpha_k)$ is a partition of $n$, we write $\chi_{\alpha}$ for the corresponding irreducible character of $S_n$ over $\mathbb{R}$, and we write $\mu_{\alpha}$ for the corresponding eigenvalue of $L_n$. Diaconis and Shahshahani derived the following useful formula.
\begin{equation} \label{eq:diaconis-formula} \mu_\alpha = {n \choose 2} - \frac{1}{2} \sum_{i=1}^k\left[(\alpha_i-i)(\alpha_i-i+1)-i(i-1)\right].\end{equation}

To analyse these eigenvalues, it is useful to consider the {\em dominance ordering}, a partial order on the set of partitions of $n$ which is defined as follows.

\begin{definition}
If $\alpha = (\alpha_1,\ldots,\alpha_k)$ and $\beta = (\beta_1,\ldots,\beta_l)$ are distinct partitions of $n$, we say that $\alpha$ is greater than $\beta$ in the {\em dominance ordering} (and we write $\alpha \rhd \beta$) if $\sum_{i=1}^{r}\alpha_i \geq \sum_{i=1}^{r} \beta_i$ for all $r \in \mathbb{N}$. (Here, $\alpha_i := 0$ for all $i > k$, and similarly $\beta_i := 0$ for all $i > l$.)
\end{definition}

Diaconis and Shahshahani observed the following.

\begin{lemma}
\label{lemma:monotone}
The eigenvalues $(\mu_\alpha)_{\alpha \vdash n}$ are monotonically non-increasing with respect to the dominance ordering on the set of partitions of $n$: if $\alpha$ and $\beta$ are partitions of $n$ with $\beta \unrhd \alpha$, then $\mu_{\beta} \leq \mu_{\alpha}$. 
\end{lemma}

Notice that $\mu_{(n)} = 0$, $\mu_{(n-1,1)} = n$, and if $\alpha \neq (n)$, then $(n-1,1) \unrhd \alpha$, so $\mu_{\alpha} \geq \mu_{(n-1,1)}$. It follows that $\mu_2(L_n) = n$. Plugging this into Theorem \ref{thm:alon-milman} yields the following, essentially due to Diaconis and Shahshahani.

\begin{corollary}
\label{corr:diaconis}
If $A \subset S_n$, then
$$|\partial A| \geq \frac{|A|(n!-|A|)}{(n-1)!}.$$
\end{corollary}

This verifies Conjecture \ref{conj:ben-efraim} when $|A| = c(n-1)!$ for some $c \in \mathbb{N}$. (Note that equality holds in Corollary \ref{corr:diaconis} when $A = \{\sigma \in S_n:\ \sigma(1) \in \{1,2,\ldots,c\}$.)

The following Corollary of Lemma \ref{lemma:monotone} and (\ref{eq:diaconis-formula}) will be useful for us.

\begin{corollary}\label{cor:cormu}
  For any $t \in \{0,1,2,\ldots,n\}$, if $\alpha$ is a partition of $n$ with $\alpha_1 \leq n-t$, then we have $(n-t,t) \unrhd \alpha$, so
  \[
\mu_{\alpha} \geq \mu_{(n-t,t)}= tn - t^2+t.
  \]
\end{corollary}

Next, we need a fact about conjugation-invariant functions.
\begin{definition}
If $f:S_n \to \mathbb{R}$, we say $f$ is a {\em class function} if it is conjugation-invariant, i.e.
$$f(\pi \sigma \pi^{-1}) = f(\sigma)\quad \forall \sigma,\pi \in S_n.$$
\end{definition}

\begin{fact}
The irreducible characters of $S_n$ over $\mathbb{R}$ are an orthonormal basis for the vector space of real-valued class functions on $S_n$, under the inner product (\ref{eq:inner-product}).
\end{fact}

We now need some facts about {\em permutation characters}. Let \(\alpha = (\alpha_{1}, \ldots, \alpha_{k})\) be a partiton of \(n\). The \emph{Young diagram} of $\alpha$ is
  an array of $n$ boxes, or `cells', having $k$ left-justified rows, where row $i$
  contains $\alpha_i$ cells. For example, the Young diagram of the partition \((3,2,2)\) is:
  \[
\yng(3,2,2)
\]

If the array contains the numbers \(\{1,2,\ldots,n\}\) inside the cells, we call it an \(\alpha\)-\emph{tableau}, or a {\em tableau of shape \(\alpha\)}; for example,
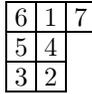
\[
\young(617,54,32)
\]
is a \((3,2,2)\)-tableau. Two \(\alpha\)-tableaux are said to be \emph{row-equivalent} if they have the same numbers in each row.

An \(\alpha\)-\emph{tabloid} is an \(\alpha\)-tableau with unordered row entries (or, more formally, a row-equivalence class of \(\alpha\)-tableaux). For example, the \((3,2,2)\)-tableau above corresponds to the following \((3,2,2)\)-tabloid:
$$\begin{array}{ccccc} \{&1&6&7&\}\\ \{&4&5&&\}\\ \{&2&3&&\}\end{array}$$
where each row is a set, not a sequence. Consider the natural left action of \(S_{n}\) on the set \(X^{\alpha}\) of all \(\alpha\)-tabloids. For example, the permutation $(1,5)(2,6,4)(3)(7)$ (written in disjoint cycle notation) acts on the tabloid above as follows:
$$(1,5)(2,6,4)(3,7) \left(\begin{array}{ccccc} \{&1&6&7&\}\\ \{&4&5&&\}\\ \{&2&3&&\}\end{array}\right) = \begin{array}{ccccc} \{&3&4&5&\}\\ \{&1&2&&\}\\ \{&6&7&&\}\end{array}$$
Let \(M^{\alpha} = \mathbb{R}[X^{\alpha}]\) be the corresponding permutation representation, i.e. the real vector space with basis \(X^{\alpha}\) and \(S_{n}\) action given by extending linearly. We write $\xi_{\alpha}$ for the character of this representation. The $\{\xi_{\alpha}\}_{\alpha \vdash n}$ are called the {\em permutation characters} of $S_n$. If $\sigma \in S_n$, then $\xi_{\alpha}(\sigma)$ is simply the number of $\alpha$-tabloids fixed by $\sigma$. 

We can express the irreducible characters in terms of the permutation characters using the {\em determinantal formula}: for any partition \(\alpha\) of \(n\),
\begin{equation}\label{eq:determinantalformula} \chi_{\alpha} = \sum_{\pi \in S_{n}} \sgn(\pi) \xi_{\alpha - \textrm{id}+\pi}.\end{equation}
Here, if \(\alpha = (\alpha_{1},\alpha_{2},\ldots,\alpha_{l})\), then \(\alpha - \textrm{id}+\pi\) is defined to be the sequence
\[(\alpha_{1}-1+\pi(1),\alpha_{2}-2+\pi(2),\ldots,\alpha_{l}-l+\pi(l)).\]
If this sequence has all its entries non-negative, then we let \(\overline{\alpha-\textrm{id}+\pi}\) be the partition of \(n\) obtained by reordering its entries, and we define \(\xi_{\alpha - \textrm{id}+\pi} = \xi_{\overline{\alpha-\textrm{id}+\pi}}\). If the sequence has a negative entry, then we define \(\xi_{\alpha - \textrm{id}+\pi} = 0\). It is easy to see that if \(\xi_{\beta}\) appears on the right-hand side of (\ref{eq:determinantalformula}), then \(\beta \unrhd \alpha\), so the determinantal formula expresses \(\chi_{\alpha}\) in terms of \(\{\xi_{\beta}: \ \beta \unrhd \alpha\}\). We may rewrite the determinantal formula as
\begin{equation}\label{eq:determinantalformula2} \chi_{\alpha} = \sum_{\beta \unrhd \alpha} c_{\alpha \beta} \xi_{\beta},\end{equation}
where $c_{\alpha \beta} \in \mathbb{Z}$ for each $\beta \unrhd \alpha$.

We need the following.
\begin{lemma}
\label{lemma:determinantal-bound}
Let $u \in \mathbb{N}$, and let $\alpha$ be a partition of $n$ with $\alpha_1=n-u$. Then
$$\sum_{\beta \unrhd \alpha}|c_{\alpha \beta}| \leq (u+1)!.$$
\end{lemma}
\begin{proof}
Fix an integer $i \geq u+2$. Note that $\alpha_{i} = 0$. If $\pi \in S_n$ with $\pi(i) < i$, then $\alpha(i)-i+\pi(i) < 0$, so $\xi_{\alpha-\textrm{id}+\pi} = 0$. Hence, for any permutation $\pi \in S_n$ such that $\xi_{\alpha-\textrm{id}+\pi} \neq 0$, we must have $\pi(i) \geq i$ for all $i \in \{u+2,u+3,\ldots,n\}$, so $\pi(i) = i$ for all $i \in \{u+2,u+3,\ldots,n\}$, i.e. $\pi \in S_{[u+1]}$. This proves the lemma.
\end{proof}

Observe that if $u \leq n/2$, then the number of partitions $\alpha$ of $n$ with $\alpha_1=n-u$ is precisely $p(u)$. Hence, if $t \leq n/2+1$, then the number of partitions $\alpha$ of $n$ with $\alpha_1 \geq n-t+1$ is precisely $\sum_{i=0}^{t-1} p(i).$ We will use the following crude bound on this quantity.

\begin{lemma}
\label{lemma:sum-bound}
If $t \in \mathbb{N}$, then
$$\sum_{i=0}^{t-1} p(i) \leq t!.$$
\end{lemma}
\begin{proof}
We have $p(i) \leq i!$ for all $i \in \mathbb{N} \cup \{0\}$, since $p(i)$ is the number of conjugacy-classes in the symmetric group $S_i$. Hence,
$$\sum_{i=0}^{t-1}p(i) \leq \sum_{i=0}^{t-1}i! \leq t!.$$
\end{proof}

\section{Proof of Theorem \ref{thm:main}}
Let $A \subset S_n$ be a conjugation-invariant set with $0 < |A|=pn! \leq n!/2$. Note that by choosing $c>0$ small enough, we may assume that $p \leq \epsilon_0$, for any absolute constant $\epsilon_0$. Indeed, if $\epsilon_0 \leq p \leq 1/2$, then by Corollary \ref{corr:diaconis}, we have
$$|\partial A| \geq \tfrac{1}{2} n \cdot |A| \geq c \cdot \frac {\log \left( \tfrac{1}{\epsilon_0} \right)
    }{\log \log \left( \tfrac{2}{\epsilon_0} \right)}\cdot n \cdot |A| \geq c \cdot \frac {\log \left( \tfrac{1}{p} \right)
    }{\log \log \left( \tfrac{2}{p} \right)}\cdot n \cdot |A|,$$
provided $c$ is chosen to be sufficiently small depending on $\epsilon_0$.

From (\ref{eq:eval-expansion}) applied to $T_n$, we have
\begin{equation} \label{eq:partial} |\partial A|=n!\la 1_A,L_n 1_A\ra.\end{equation}
Since the irreducible characters of $S_n$ over $\mathbb{R}$ are an orthonormal basis for the space of real-valued class functions on $S_n$, we may write
\begin{equation}\label{eq:linear-combination} 1_{A} = \sum_{\alpha \vdash n} w_{\alpha} \chi_{\alpha},\end{equation}
where $w_{\alpha} = \langle 1_{A},\chi_{\alpha} \rangle \in \mathbb{R}$ for each $\alpha \vdash n$. Since $\chi_{\alpha}$ is an eigenvector of $L_n$ with eigenvalue $\mu_{\alpha}$, substituting this into (\ref{eq:partial}) gives
\begin{equation}\label{eq:exact} |\partial A|=n!\sum_{\alpha \vdash n} \mu_\alpha w_\alpha^2.\end{equation}
Now define $K=K(p)$ by
\begin{equation}\label{eq:K-defn} K^{2K}=\frac 1{p}.\end{equation}
We pause to note simple lower and upper bounds on $K$. Taking the logarithm of both sides of (\ref{eq:K-defn}) gives:
  \[
  2K\log K = \log \tfrac 1p \leq \log \tfrac 2p.
  \]
  Thus $K< \log \tfrac 2p$, which implies $\log K < \log\log \frac
  2p$. Therefore:
  $$
  K = \frac {\log \tfrac 1p}{2\log K} \ge \frac {\log {\frac 1p}
  }{2\log \log \frac 2p}.
  $$

On the other hand, since $A \neq \emptyset$, we have $p \geq 1/n!$, so
$$K^{2K} \leq n! \leq n^n,$$
and therefore $K \leq n$. Putting these two bounds together, we have
\begin{equation} \label{eq:lower-bound-K}
 \frac {\log {\frac 1p}
  }{2\log \log \frac 2p} \leq K \leq n.
  \end{equation}

 Let $M$ be a large, fixed integer. (For concreteness, we may take $M=18$.) Define $t_p = \lfloor K/M \rfloor$; note that
 \begin{equation} \label{eq:t-bound} 
 \frac {\log {\frac 1p}
  }{4M\log \log \frac 2p} \leq t_p \leq \frac{n}{M}.\end{equation}
  
 We need the following bound on $|w_{\alpha}|$ for $\alpha_1 > n-t_p$.
\begin{proposition}
\label{prop:w-bound}
Let $\alpha \vdash n$ with $\alpha_1 = n-t$, where $t < t_p$. Then
$$|w_{\alpha}| \le \frac 1{K^{2K(1-\frac 8M)}}.$$
\end{proposition}
\begin{proof}
By (\ref{eq:linear-combination}) and the orthonormality of $\{\chi_{\alpha}\}_{\alpha \vdash n}$, we have
\begin{equation}\label{eq:expansion} w_{\alpha} = \la 1_{A},\chi_{\alpha} \ra = \frac{1}{n!}\sum_{\sigma \in A} \chi_{\alpha}(\sigma)\end{equation}
By (\ref{eq:determinantalformula2}), we have
\begin{equation}\label{eqn:boundchi}
|\chi_\alpha(\sigma)| \le \sum_{\beta\unrhd\alpha}
|c_{\alpha\beta}|\xi_\beta(\sigma).
\end{equation}
For each tabloid $T$ of shape $\beta=(\beta_1,\dots,\beta_n)$ let $\tilde T$ be the tabloid of shape $\tilde \beta :=
(\beta_1,n-\beta_1)$ obtained by collapsing all the rows
other than the first one into a single row of length $n-\beta_1 =
\sum_{i=2}^n\beta_i$. We observe that for every permutation
$\sigma\in S_n$, if $\sigma$ fixes the tabloid $T$, then it also fixes
the tabloid $\tilde T$. Moreover, the mapping $T\mapsto \tilde T$ is a
surjection, and is at most $(n-\beta_1)!$ to $1$. Recalling that $\xi_{\beta}(\sigma)$ is the number of $\beta$-tabloids fixed by $\sigma$, we obtain
\[
\xi_\beta(\sigma) \le (n-\beta_1)!\cdot \xi_{\tilde\beta} (\sigma).
\]
Substituting this into \eqref{eqn:boundchi} yields
\begin{equation}\label{eq:mod}
|\chi_\alpha(\sigma)| \le \sum_{\beta\unrhd\alpha} |c_{\alpha\beta}|
\cdot (n-\beta_1)!\cdot \xi_{\tilde \beta} (\sigma).
\end{equation}
Substituting this bound into (\ref{eq:expansion}), we obtain
\[
|w_\alpha| \le \sum_{\beta \unrhd \alpha}|c_{\alpha\beta}|
(n-\beta_1)!\cdot\frac 1{n!}\sum_{\sigma\in A} \xi_{\tilde \beta} (\sigma).
\]

We need the following lemma (which we
will prove in Section \ref{sec:lemmaproof}).
\begin{lemma} \label{lemma:main} Let $A\subset S_n$ be a family of
  permutations with $|A|=p n!$ and let $s \in \mathbb{N}$ with $s<t_p$. Then
  \begin{equation*}
    \tfrac 1{n!} \sum_{\sigma\in A}\xi_{(n-s,s)}(\sigma) \le
    \frac 1{K^{2K(1-\frac 7M)}}.
  \end{equation*}
\end{lemma}

Now define $s(\beta) = n-\beta_1$. Notice
that $\beta \unrhd \alpha$ implies $s\le t$, so from (\ref{eq:mod}) and Lemma \ref{lemma:main} we obtain
\begin{align*}
  |w_\alpha| &\le \sum_{\beta\unrhd\alpha} |c_{\alpha\beta}| \cdot
  s!\cdot \frac 1{n!}\sum_{\sigma\in A}\xi_{(n-s,s)} (\sigma) \\ &\le
  \sum_{\beta\unrhd\alpha} |c_{\alpha\beta}| \cdot s!\cdot \frac
  1{K^{2K(1-\frac 7M)}} \\
  &\le \frac {t_p!}{K^{2K(1-\frac 7M)}} \cdot \sum_{\beta\unrhd\alpha}
  |c_{\alpha\beta}|.
\end{align*}

By Lemma \ref{lemma:determinantal-bound}, we have
\[
\sum_{\beta\unrhd\alpha} |c_{\alpha\beta}| \le t_p!,
\]
and therefore
\[
|w_\alpha| \le \frac {(t_p!)^2}{K^{2K(1-\frac 7M)}} \le  \frac
{t_p^{2t_p}}{K^{2K(1-\frac 7M)}} \le \frac 1{K^{2K(1-\frac 8M)}},
\]
proving Proposition \ref{prop:w-bound}.
\end{proof}
Using (\ref{eq:exact}), Corollary \ref{cor:cormu}, Lemma \ref{lemma:sum-bound} and Proposition \ref{prop:w-bound}, we obtain:
\begin{align*}
  |\partial A| & = n! \sum_{\alpha \vdash n}\mu_\alpha w_\alpha^2\\
  & \geq n! \sum_{\alpha_1 \leq n-t_p}\mu_\alpha w_\alpha^2\\
  & \ge n!\cdot\mu_{(n-t_p,t_p)}\cdot\sum_{\alpha_1 \leq n-t_p} w_\alpha^2
  \\
  &= n!\cdot(n\cdot t_p - (t_p)^2+t_p)\cdot \left(\|1_A\|^2 -
    \sum_{\alpha_1 > n-t_p} w_\alpha^2 \right) \\
  &\ge n!\cdot(n\cdot t_p- t_p^2)\cdot\left(p - \left|\{\alpha\vdash
      n : \alpha_1 > n-t_p\}\right| \cdot \left(\max_{\alpha_1 > n-t_p}|w_{\alpha}|\right)^2\right) \\
  &\ge n!\cdot(n\cdot t_p- t_p^2)\cdot\left(p - t_p! \cdot
    \left(\frac 1{K^{2K(1-\frac 8M)}}\right)^2\right) \\
  & \ge n! \cdot (n\cdot t_p-t_p^2)\cdot\left(p - \frac {{K}^{\frac
      KM}}{K^{2K(2-\frac {16}M)}}\right) \\
    & \ge n! \cdot (n\cdot t_p-t_p^2)\cdot \left(p - \frac {1}{K^{2K(2-\frac {17}M)}}\right) \\
  & \ge n! \cdot (n\cdot t_p-t_p^2)\cdot (p -
  p^{2-\frac {17}{M}})
\end{align*}
Taking $M=18$, and using (\ref{eq:t-bound}), we have $t_p \leq n/M = n/18$. Hence,
\begin{align*}
  |\partial A| &\geq n! \cdot (n\cdot t_p-t_p^2)\cdot p(1-p^{1/18}) \geq c_0 \cdot t_p \cdot n \cdot |A|,
\end{align*}
for some absolute constant $c_0 >0$.

Using (\ref{eq:t-bound}), it follows that
$$|\partial A| \geq c \cdot \frac {\log \left( \tfrac 1{p} \right)
    }{\log \log \left( \tfrac 2{p} \right)}\cdot n \cdot |A|$$
    for some absolute constant $c>0$, proving Theorem \ref{thm:main}.

\section{Proof of Lemma \ref{lemma:main}} \label{sec:lemmaproof}
Let $\sigma \in S_n$. For $i \in [n]$, let
$C_i(\sigma)$ denote the number of cycles in $\sigma$ of length
$i$. Since $\xi_\alpha(\sigma)$ is the number of tabloids of shape
$\alpha$ fixed by $\sigma$, it follows that $\xi_{(n-s,s)}(\sigma)$ is the number of
subsets of $[n]$ of size $s$ that are fixed by
$\sigma$. Observe that if a set $S \subset [n]$ with $|S|=s$
is fixed by $\sigma$, then $S$ is a union of at most $s$ cycles of $\sigma$, all of which have length at most $s$. Therefore,
\begin{equation} \label{eq:sum_i}
  \xi_{(n-s,s)}(\sigma) \le \left( \sum_{i=1}^s
    C_i(\sigma)\right)^s
    \leq s^{s-1}\cdot \sum_{i=1}^s (C_i(\sigma))^s,
\end{equation}
using Jensen's inequality.

Hence,
\begin{align*}
  \tfrac 1{n!} \sum_{\sigma\in A}\xi_{(n-s,s)}(\sigma)
  \le \frac{ s^{s-1} }{n!} \sum_{i=1}^s \sum_{\sigma \in A} \left(
    C_i(\sigma) \right)^s.
\end{align*}

For $i \in [n]$ and $1\le j \le \lfloor \frac ni \rfloor$, we define
\[
D_{n,i,j}=\{ \sigma\in S_n:C_i(\sigma)=j\}.
\]

It was shown in \cite{goncharov} (cf. \cite{arratia}) that
\begin{align*}
  |D_{n,i,j}| = \frac{n!i^{-j}}{j!}\sum_{l=0}^{\lfloor n/i\rfloor -j}
  (-1)^l\frac{i^{-l}}{l!},
\end{align*}
which implies
\begin{align}\label{eq:bound_D_size}
  \tfrac 13 \cdot \frac {n!}{i^jj!} \le |D_{n,i,j}| \le \frac{n!}{i^{j}j!},
\end{align}
unless $i=1$ and $j=n-1$, in which case $|D_{n,i,j}|=0$.

Fix a specific $i \in [s]$, and let $\kappa = \kappa_p(i) \in \mathbb{R}$ be such
that
\begin{align}\label{eq:defk}
i^{\kappa}\cdot \kappa^{\kappa} = \frac{1}{p}.
\end{align}
Define $k=k_p(i)=\lfloor \kappa_p(i) \rfloor$. Clearly, for fixed $p$,
$k_p(i)$ is monotone non-increasing in $i$, and therefore $k_p(i) \geq K-1$ for all
$1\le i \le s$.

Define
$$D_{n,i,\ge k} := \bigcup_{j=k}^{\lfloor \frac ni \rfloor} D_{n,i,j}.$$
From the left-hand side of
\eqref{eq:bound_D_size}, it follows that
\begin{align} \label{eq:d_greater_a}
|D_{n,i,\ge k}| \ge |D_{n,i,k}| \ge \tfrac 13
\frac {n!}{i^kk!} \ge  \frac {n!}{i^k k^k}  \ge \frac {n!}{i^{\kappa} \kappa^{\kappa}} = |A|,
\end{align}
using the fact that $n-2 \geq k \geq K-1 \geq 3$. Hence,
\begin{align*}
  \frac{ s^{s-1} }{n!}\sum_{\sigma \in A} \left(C_i(\sigma)
    \right)^s
  &\le \frac{ s^{s-1} }{n!}\left( \sum_{j=k}^{\lfloor \frac ni \rfloor}|D_{n,i,j}\cap A|\cdot j^s +
  |A\setminus D_{n,i,\ge k}|\cdot k^s \right).
\end{align*}
Equation \eqref{eq:d_greater_a} implies that $|A \setminus D_{n,i,\ge
  k}| \le |D_{n,i,\ge k}\setminus A|$. Hence,
\begin{align*}
  \frac{ s^{s-1} }{n!}\sum_{\sigma \in A} \left(C_i(\sigma) \right)^s
  & \le \frac{ s^{s-1} }{n!}\left( \sum_{j=k}^{\lfloor \frac ni
      \rfloor}|D_{n,i,j}\cap A|\cdot j^s + |D_{n,i,\ge k} \setminus
    A|\cdot k^s \right)\\
    & \le \frac{ s^{s-1} }{n!}\left(
    \sum_{j=k}^{\lfloor \frac ni \rfloor}|D_{n,i,j}|\cdot j^s \right).
\end{align*}
Using the right-hand inequality of \eqref{eq:bound_D_size}, we obtain
\begin{align*}
  \frac{ s^{s-1} }{n!}\sum_{\sigma \in A} \left(C_i(\sigma) \right)^s
  \le s^{s-1} \sum_{j=k}^{\lfloor \frac ni \rfloor} \frac
  {j^s}{i^jj!} \le \frac {s^{s-1}}{i^k} \sum_{j=k}^{\lfloor \frac ni \rfloor} \frac
  {j^s}{j!} \le \frac {s^{s-1}}{i^k} \sum_{j=k}^{\infty} \frac
  {j^s}{j!}.
\end{align*}
To bound the latter sum, we make the following claim.
\begin{claim} \label{claim:uglyfactorial} For any $s,k\in \N$
  such that $Ms\le k $ and $e^M\le k$, we have
  \[
  \sum_{j=k}^\infty \frac {j^s}{j!} \le \frac 1{k^{(1-\frac 3M)k}}.
  \]
\end{claim}
\begin{proof}[Proof of Claim \ref{claim:uglyfactorial}]
  Using Stirling's bound $j! \geq (j/e)^j$ (valid for all $j \in \mathbb{N}$), we obtain
  \begin{align*}
    \sum_{j=k}^\infty\frac{j^s}{j!} \le \sum_{j=k}^\infty \frac
    {j^se^j}{j^j}.
  \end{align*}
  The fact that $s\le\frac 1M j$ implies $j^s \le j^{\frac 1Mj}$, whereas
  the fact that $2M<e^M\le k \le j$ implies that $j \geq 2M$ and $e^j\le j^{\frac jM}$. Therefore,
  \[
  \sum_{j=k}^\infty \frac {j^se^j}{j^j} \le \sum_{j=k}^\infty\frac
  1{j^{j-\frac 3Mj+2}} \le \frac 1{k^{(1-\frac 3M)k}}\sum_{j=k}^\infty
  \frac 1{j^2}
   \le \frac 1{k^{(1-\frac 3M)k}},
  \]
  proving the claim.
\end{proof}
Coming back to the proof of Lemma \ref{lemma:main}, note from (\ref{eq:lower-bound-K}) that we have $k \geq \kappa-1 \geq K-1 \geq e^M$ provided $p \leq \epsilon_0$ and $\epsilon_0$ is a sufficiently small absolute constant. Hence, we may apply Claim \ref{claim:uglyfactorial}, giving
\[
\frac{ s^{s-1} }{n!}\sum_{\sigma \in A} \left(C_i(\sigma) \right)^s \le
\frac {s^{s-1}}{i^k} \sum_{j=k}^\infty\frac{j^s}{j!} \le \frac{
  s^{s-1}}{i^k k^{(1-\frac 3M)k}} \le
\frac{s^{s-1}}{i^{\kappa}\kappa^{(1-\frac 6M)\kappa}},
\]
where the last inequality follows from the fact that $k\ge \kappa -1$. Equation \eqref{eq:defk} implies $i^{\kappa}\kappa^{\kappa} = K^{2K}$,
which gives
\[
\frac{ s^{s-1} }{n!}\sum_{\sigma \in A} \left(C_i(\sigma) \right)^s \le
\frac{s^{s-1}}{i^{\kappa}\kappa^{(1-\frac 6M)\kappa}} \le
\frac{s^{s-1}}{i^{(1-\frac 6M)\kappa}\kappa^{(1-\frac 6M)\kappa}} =
\frac {s^{s-1}}{K^{2K(1-\frac 6M)}}.
\]

Plugging this into equation \eqref{eq:sum_i} gives
\[
  \tfrac 1{n!} \sum_{\sigma\in A}\xi_{(n-s,s)}(\sigma)
  \le \frac{ s^{s-1} }{n!} \sum_{i=1}^s \sum_{\sigma \in A} \left(
    C_i(\sigma) \right)^s \le \frac {s^{s}}{K^{2K(1-\frac 6M)}} \le
  \frac 1{K^{2K(1-\frac 7M)}},
\]
proving Lemma \ref{lemma:main}.

\section{Conclusion}
For fixed each pair of positive integers $n,k$ such that there exists a conjugation-invariant subset of $S_n$ with size $k$, define
$$\Xi_n(k) = \min\{|\partial A|:\ A \subset S_n,\ A \textrm{ is conjugation invariant},\ |A|=k\}.$$
We have given a lower bound on $\Xi_n(k)$ which is sharp up to an absolute constant factor. It would be interesting to determine more accurately the behaviour of $\Xi_n(k)$. We make the following conjecture in this regard.
\begin{conjecture}
Let $n,k$ be positive integers such that there exists a conjugation-invariant subset of $S_n$ with size $k$. Let $s = s(n,k) \in \mathbb{N}$ be such that
$$|A_{s}| \leq k \leq |A_{s-1}|,$$
where $A_j$ is defined as in (\ref{eq:fixed-points}) for each $j \in \{0,1,\ldots,n\}$. Then
$$\Xi_n(k) \geq \min\{|\partial (A_{s-1})|,|\partial (A_{s})|\}.$$
\end{conjecture}
At present, our methods do not seem capable of proving such an exact result.

\section{Appendix}
For completeness, we give here a proof of the bound (\ref{eq:unrestricted-bound}) stated in the Introduction. First, we need a small amount of additional notation. If $i_1,\ldots,i_r \in [n]$ are distinct and $j_1,\ldots,j_r \in [n]$ are distinct, we write
$$R_{i_1 \mapsto j_1,i_2 \mapsto j_2 ,\ldots, i_r \mapsto j_r} := \{\sigma \in S_n:\ \sigma(i_k) = j_k\ \forall k \in [r]\}.$$
If $G = (V,E)$ is a finite graph and $S \subset V$, we write $G[S]$ for the subgraph of $G$ induced on the set of vertices $S$, that is, the graph with vertex-set $S$, where $vw$ is an edge of $G[S]$ if and only $vw$ is an edge of $G$, for each $v,w \in S$. Moreover, if $S,T \subset V$ with $S \cap T = \emptyset$, we write $e(S,T)$ for the number of edges of $G$ between $S$ and $T$.

\begin{proposition}
Let $t \in \{0,1,\ldots,n-1\}$, and let $A \subset S_n$ be an initial segment of the lexicographic ordering on $S_n$, with $(n-t-1)! < |A| \leq (n-t)!$. Then
$$|\partial A| \leq (t+3/2)(n-1)|A|.$$
\end{proposition}
\begin{proof}
The case of general $t$ will follow from the case $t=0$, which we deal with in the following claim.
\begin{claim}
\label{claim:large}
Let $A \subset S_n$ be an initial segment of the lexicographic ordering on $S_n$ with $|A| > (n-1)!$. Then $|\partial A| \leq \tfrac{3}{2} (n-1)|A|$.
\end{claim}
\begin{proof}[Proof of claim.]
By induction on $n$. The claim holds trivially for $n \leq 2$. Let $n \geq 3$, and assume the claim holds for $n-1$. Let $A \subset S_n$ be an initial segment of the lexicographic ordering on $S_n$ with $|A| > (n-1)!$. Then we may write
$$A = R_{1 \mapsto 1} \cup R_{1 \mapsto 2} \cup \ldots \cup R_{1 \mapsto j-1} \cup A_1,$$
where $A_1 \subset R_{1 \mapsto j}$ and $j \in \{2,3,\ldots,n\}$. Define
$$A_0 = R_{1 \mapsto 1} \cup R_{1 \mapsto 2} \cup \ldots \cup R_{1 \mapsto j-1};$$
then $A = A_0 \dot \cup A_1$. Define
$$C = R_{1 \mapsto j+1} \cup \ldots \cup R_{1 \mapsto n}.$$
Notice that
$$|\partial A| = e(A,C) + e(A_0,R_{1 \mapsto j} \setminus A_1) + e(A_1, R_{1 \mapsto j} \setminus A_1).$$
Observe that
$$e(A,C) + e(A_0,R_{1 \mapsto j} \setminus A_1) \leq (n-1)|A|,$$
since each edge on the left-hand side is between $T_{1 \mapsto i}$ and $T_{1 \mapsto i'}$ for some $i \neq i'$, and there are at most $n-1$ such edges of $T_n$ incident with any permutation. To complete the proof of the inductive step, it suffices to show that
$$e(A_1, R_{1 \mapsto j} \setminus A_1) \leq \tfrac{1}{2}(n-1)|A|.$$
In fact, we prove the slightly stronger bound
$$e(A_1, R_{1 \mapsto j} \setminus A_1) \leq \tfrac{1}{2}(n-2)|A|,$$
by splitting into two cases.\newline

Case (i): $\min\{|A_1|,|R_{1 \mapsto j} \setminus A_1|\} \leq (n-2)!$.

Notice that $e(A_1,R_{1 \mapsto j} \setminus A_1)$ is simply the size of the edge boundary of $A_1$ in the graph $T_n[R_{1 \mapsto j}]$, which is isomorphic to the transposition graph $T_{n-1}$. This graph is ${n-1 \choose 2}$-regular, so trivially,
$$e(A_1,R_{1 \mapsto j} \setminus A_1) \leq {n-1 \choose 2} \min\{|A_1|,|T_{1\mapsto j} \setminus A_1|\}.$$
Since $\min\{|A_1|,|R_{1 \mapsto j} \setminus A_1|\} \leq (n-2)! < |A|/(n-1)$, we have
$$e(A_1,R_{1 \mapsto j} \setminus A_1) \leq \tfrac{1}{2}(n-2)|A|.$$
This completes the inductive step in case (i).\newline

Case (ii): $\min\{|A_1|,|R_{1 \mapsto j} \setminus A_1|\} > (n-2)!$.

Define
$$B = \begin{cases} A_1 & \mbox{if } |A_1| \leq \tfrac{1}{2}(n-1)!;\\ R_{1 \mapsto j} \setminus A_1 & \mbox{otherwise.}\end{cases}$$
Then $|B| > (n-2)!$, and $e(B,R_{1 \mapsto j} \setminus B)$ is the size of the edge-boundary of $B$ in the graph $T_n[R_{1 \mapsto j}]$. Hence, by the induction hypothesis, we have
$$e(B,R_{1 \mapsto j} \setminus B) \leq \tfrac{3}{2}(n-2)|B|.$$
Since $|B| \leq \tfrac{1}{3}|A|$, we have
$$e(B,R_{1 \mapsto j} \setminus B) \leq \tfrac{1}{2}(n-2)|A|.$$
Hence,
$$e(A_1,R_{1 \mapsto j} \setminus A_1) \leq \tfrac{1}{2}(n-2)|A|.$$
This completes the inductive step in case (ii), proving the claim.
\end{proof}
We can now prove the proposition for $t \in [n-1]$. Let $t \in [n-1]$ and let $A \subset S_n$ be an initial segment of the lexicographic ordering on $S_n$ with $(n-t-1)! < |A| \leq (n-t)!$. Then $A \subset R_{1 \mapsto 1,\ldots,t \mapsto t}$. Hence,
\begin{equation} \label{eq:boundary-sum} |\partial A| = e(A,R_{1 \mapsto 1,\ldots,t \mapsto t} \setminus A) + e(A, S_n \setminus R_{1 \mapsto 1,\ldots, t \mapsto t}).\end{equation}
Observe that
\begin{equation}\label{eq:bound-outside} e(A, S_n \setminus R_{1 \mapsto 1,\ldots, t \mapsto t}) \leq t(n-1)|A|,\end{equation}
since each $\sigma \in R_{1 \mapsto 1,\ldots, t \mapsto t}$ has exactly $t(n-1)$ neighbours in $S_n \setminus R_{1 \mapsto 1,\ldots, t \mapsto t}$. Moreover, $e(A,R_{1 \mapsto 1,\ldots,t \mapsto t} \setminus A)$ is simply the size of the edge boundary of $A$ in the graph $T_n[R_{1 \mapsto 1, \ldots,t \mapsto t}]$, which is isomorphic to the transposition graph on $S_{n-t}$. Hence, by Claim \ref{claim:large},
\begin{equation} \label{eq:bound-inside} e(A,R_{1 \mapsto 1,\ldots,t \mapsto t} \setminus A) \leq \tfrac{3}{2} (n-t-1) |A|.\end{equation}
Plugging (\ref{eq:bound-outside}) and (\ref{eq:bound-inside}) into (\ref{eq:boundary-sum}) gives
$$|\partial A| \leq t(n-1)|A| + \tfrac{3}{2} (n-t-1) |A| \leq (t+3/2)(n-1) |A|.$$
This completes the proof of the proposition.
\end{proof}

\subsubsection*{Ackowledgement}
We would like to thank Itai Benjamini for several helpful discussions, and an anonymous referee for several helpful suggestions.

\end{document}